\documentclass[journal]{IEEEtran}
\usepackage{graphicx,times,amsmath,color,amssymb,cite,amsthm}
\usepackage{caption}
\usepackage{subcaption}
\newtheorem{theorem}{Theorem}
\newtheorem{lemma}{Lemma}

\newtheorem{remark}{Remark}
\newtheorem{assumption}{Assumption}

\newtheorem{problem}{Problem}

\pdfoutput=1
\renewcommand{\baselinestretch}{1}
\usepackage[a4paper,left=1in,right=1in,top=1in,bottom=1in]{geometry}

\def\aa{{\alpha}}
\def\ee{{\epsilon}}

\def\be{\begin{equation}}
\def\bea{\begin{eqnarray}}
\def\beas{\begin{eqnarray*}}
\def\eea{\end{eqnarray}}
\def\eeas{\end{eqnarray*}}

\def\bi{\begin{itemize}}
\def\ee{\end{equation}}
\def\ei{\end{itemize}}

\def\z1{z^{-1}}
\def\la{\label}

\def\mR{\mathbb{R}}

\ifCLASSINFOpdf
  \else
 \fi

\begin{document}
\title{Maximum Likelihood Localization of Radiation Sources with Unknown Source Intensity}

\author{Henry~E.~Baidoo-Williams,~\IEEEmembership{Member,~IEEE}
        %Soura~Dasgupta,~\IEEEmembership{Fellow,~IEEE,}
        %Er-Wei~Bai,~\IEEEmembership{Fellow,~IEEE,}
        %and~Raghu~Mudumbai,~\IEEEmembership{Member,~IEEE}% <-this % stops a space
\thanks{H. E. Baidoo-Williams was with the US Army Research Laboratory, Aberdeen Proving Ground, MD and 
University at Buffalo, SUNY, NY, 14150, USA e-mail: (henrybai AT buffalo DOT edu).}% <-this % stops a space
%\thanks{S. Dasgupta, E. Bai and R. Mudumbai are with University of Iowa, Iowa City, IA, 52242, USA.}% <-this % stops a space
%\thanks{Manuscript received June 20, 2016; revised --, --.}
}

\markboth{}%
{Baidoo-Williams\MakeLowercase{\textit{et al.}}: Maximum Likelihood Localization of Radiation Sources with unknown Source Intensity}
\maketitle

\begin{abstract}
In this paper, we consider a novel and robust maximum likelihood approach to localizing radiation sources with unknown statistics of the source signal strength. The result utilizes the smallest number of sensors required theoretically to localize the source.  It is shown, that should the source lie in the open convex hull of the sensors, precisely $N+1$ are required in $\mathbb{R}^N, ~N \in \{1,\cdots,3\}$. It is further shown that the region of interest, the open convex hull of the sensors,  is entirely devoid of false stationary points. An augmented gradient ascent algorithm with random projections should an estimate escape the convex hull is presented. 
\end{abstract}


\IEEEpeerreviewmaketitle

\section{Introduction}
\IEEEPARstart{P}{revious} work  on localization of radioactive sources have shown that using exactly $N+1$ sensors in $\mathbb{R}^N$ suffices to localize if the source is inside the open convex hull of the measurement sensors \cite{BDBM},\cite{BWT},\cite{Bai},\cite{BML},\cite{GML},\cite{VK}. This result however makes a key assumption -- that statistics of the source is available. In practice, however, it is highly unlikely for the signal source strength to be known a priori. Subsequent research on unknown source signal strength however have presented results that lead to multiple local maxima, leading to the algorithms performance depending on initial estimates \cite{BML}. The multiple local maxima problem is resolved in \cite{BML} by requiring the initial estimate being in the basin of attraction of the global maximum. In contrast, in this paper, a novel maximum likelihood approach which removes that requirement of the initial estimate being in the basin of attraction is presented. The result also shows that the region of interest has a unique maximum.

The impact of our result needs little motivation. We however provide a little motivation here. Consider a hostile actor with a radioactive material, through a port of entry. It is reasonable to assume unknown a priori information about the radiation source. Further, by shear economy of scale, it is desirable to use the minimum number of sensors possible across all ports of entries and all monuments of intrinsic value. In \cite{BML}, 19 such sensors are required in $2-D$ as compared to 3 in this work. In this regard, we extend the works in \cite{BDBM},\cite{BWT} and \cite{BML} to ``{\it unknown source intensity}" using the ``{\it smallest number}" of sensors possible and show that using a novel non-concave maximum likelihood based profit function, the region of interest is without a  false stationary point and has a unique global maximizer. 

The key contributions of this paper are as follows:
\begin{enumerate}
\item A novel approach to a profit function which ensures a unique maximizer is presented.
\item The smallest possible number of sensors are used for localization with uniqueness of solution guaranteed.
\item The maximum likelihood based algorithm is independent of initialization; the requirement of initialization in the basin of attraction is removed entirely.
\item A robust gradient ascent algorithm which achieves global uniform asymptotic convergence in probability is presented.
\end{enumerate}

The presentation of the rest of the paper will be organized as follows: In section \ref{sps}, the problem statement is formalized. In section \ref{scost}, a novel profit function is developed along with its gradient ascent algorithm.  In section \ref{smain}, the main results and its implications are presented. Robust simulations results are presented in section \ref{ssim}. Section \ref{scon} concludes the paper.

\section{Problem Statement} \la{sps}
Consider a  stationary radiation source located at ${\pmb y^*}\in \mR^N$, $N\in \{1,\cdots,3\}$, and $n$ measurement sensors located at ${\pmb x_i}\in \mR^N$, $N \in \{1,\cdots,3\}, i \in \{1, \cdots, n\}$. Define the distance from a measurement sensor $i$ to a radiation source $\{{\pmb y},{\pmb y^*}\}$ as 
\be \la{dist}
d_{i}=\|{\pmb x_i}- {\pmb y}\|, ~ d_{i}^*=\|{\pmb x_i}-{\pmb y^*}\|
\ee
Here, $\|\cdot\|$ denotes the 2-norm, boldfaced variables denote vectors, ${\pmb y^*}$ is the true source and ${\pmb y}$ is the estimate. The sensors measure the total gamma-ray counts received thus:
\be \la{rss}
s_i=\lambda_{i} +w_i.
\ee
In (\ref{rss}), $\lambda_{i} \sim Poisson \left(\lambda_i^*= \frac{A_i^* e^{-\aa_{i} d_{i}^*}}{d_{i}^{*2}}\right)$. $A_i^*$ is the unknown signal source intensity and its value is dependent on a number of variables; type of isotope, geometric shape, total volume of the source isotope as well as the characteristics and quality of the measurement sensors \cite{Bai}. $\aa_{i}$ is the attenuation coefficient and is a function of the density of the shielding material, if any is utilized, and the propagation channel characteristics. $w_i$, which is the background noise,  is also Poisson distributed. 

Suppose the following assumptions hold:
\begin{assumption} \la{pr0}
The source location, ${\pmb y^*}$, is  distinct from the sensor locations $x_i, i \in \{1, \cdots, n\}$.
\end{assumption} 
\begin{assumption} \la{pr1}
The source intensities $A_i^*$, $i \in \{1,\cdots,n\}$  is homogenous across all radiation sensors and denoted $A^*$.
\end{assumption}
\begin{assumption} \la{pr2}
The shielding coefficients $\aa_{i}$, $i \in \{1,\cdots,n\}$, are homogenous across all radiation sensors and denoted $\aa$.
\end{assumption}
Under assumptions \ref{pr1} and \ref{pr2},  $\lambda_{i}^*$ reduces to:
\be \la{rate}
\lambda_{i}^* =\frac{A^* e^{-\aa d_{i}^*}}{d_{i}^{*2}}
\ee
\begin{assumption} \la{amon}
For $d_{i}^*>0$, $i\in \{1,\cdots,n\}$, $\lambda_{i}^*$ is strictly decreasing and analytic. Further for every $\rho_{1}, \rho_{2} >0$, there exists $M\left(\rho_1,\rho_2\right)$ such that the $\lambda_{i}^*$ and their  first two derivatives    are all bounded in magnitude by $M\left(\rho_1,\rho_2 \right)$ whenever
%\be \la{gbd}
$\rho_1 \leq d_{i}^* \leq \rho_2$.
%\ee
\end{assumption}
It is realized that given $\aa > 0$,
\be \la{orig}
z_i=A \frac{e^{-\aa d_{i}}}{d_{i}^2}
\ee
obeys Assumption \ref{amon}. This is made clearer by noting that the gradient of (\ref{orig}) is:
\be \la{orig_der}
\dot{z}_i=-z_i \frac{1}{d_{i}} \left(\aa +\frac{2}{d_{i}} \right)
\ee
which is strictly decreasing for $\aa \ge 0$. The derivative is with respect to $d_i$. Therefore $\lambda_i^*$ is strictly decreasing given $\aa \ge 0,  i \in \{1,\cdots,n\}$.
\begin{assumption} \la{an}
The number of sensors $n$, precisely equals $N+1$ and the sensors located at ${\pmb x_i}\in \mR^N$ do not lie on an  $(N-1)$-dimensional hyperplane. 
\end{assumption}
From here on, the notation, $co\{{\pmb x_1},\cdots,{\pmb x_{N+1}}\}$ will denote the {\it open}  convex hull of location of the measurement sensors, ${\pmb x_i}, i \in \{1,\cdots,N+1\}$.
\begin{assumption} \la{aconv}
The source  location ${\pmb y^*}\in \mR^N$ is in $co\{{\pmb x_1},\cdots,{\pmb x_{N+1}}\}$.
\end{assumption}
\begin{assumption} \la{lda}
Under noise free case $s_i =z_i$.
\end{assumption}
\begin{remark}
Assumption 7 is added because in our analysis we show that under noise free case, the profit function is globally exponentially convergent using gradient ascent algorithm. Further,  nonlinear literature show that global exponential convergence under noise free case ensures that in the presence of noise there is a graceful performance in estimation accuracy degradation \cite{grandstrand:2004}.
\end{remark}
\begin{problem} \la{ps}
Under Assumptions \ref{pr0} -- \ref{lda} , given the sensor readings $s_i, \in \{1, \cdots, N+1\}$, estimate the maximum likelihood of ${\pmb y^*}$. Further, is the maximum likelihood estimate of ${\pmb y^*}$ unique?
\end{problem} 

\section{The profit function and its gradient ascent  maximization} \la{scost}
Consider an observation model which in the noise free case obeys: 
\be \la{obs}
z_i=A \frac{e^{-\aa d_{i}}}{d_{i}^2}, i \in \{1,\cdots,N+1\},~ N>0
\ee
The joint density function of the received signal strengths, $s_i$, $i\in \{1,\cdots,N+1\}$ is derived as follows\cite{GML},\cite{VK}:
\begin{equation*}
\begin{split}
&f_{s_1,\cdots,s_{N+1}}(s_1,\cdots,s_{N+1}) =\prod\limits_{i=1}^{N+1}\frac{\lambda_i^{s_i} e^{-\lambda_i}}{s_i!}\\
&\log{f_{s_1,\cdots,s_{N+1}}(s_1,\cdots,s_{N+1})}=\\
&\sum\limits_{i=1}^{N+1} s_i \log \lambda_i -\lambda_i -\log s_i!\\
\end{split}
\end{equation*}
Without loss of generality, the log-likelihood function can be modified by removing the constant term to: 
\be \la{ml}
L(A,{\pmb y}) =\sum\limits_{i=1}^{N+1} s_i \log \lambda_i -\lambda_i
\ee
From (\ref{ml}), the maximum likelihood estimate of $\{A,{\pmb y}\}$, denoted  $\{\hat{A},\hat{{\pmb y}}\}$ can be derived as:
\begin{align} \la{Aml}
\hat{A} &=\frac{\sum\limits_{i=1}^{N+1} s_i}{\sum\limits_{i=1}^{N+1} \frac{e^{-\aa \hat{d}_i}}{\hat{d}_i^2}}
\end{align}
and 
\begin{align} \la{yml}
0 &=\left. \sum\limits_{i=1}^{N+1} \left( \frac{s_i \hat{d}_i^2}{e^{-\aa \hat{d}_i}}-\hat{A}\right) \left(\aa + \frac{2}{\hat{d}_{i}} \right)\frac{e^{-\aa \hat{d}_{i}}}{\hat{d}_{i}^3} \left({\pmb x}_i-{\hat{\pmb y}} \right)\right .
\end{align}
Now consider the profit function:
\be \la{cost}
J({\pmb y})=\sum\limits_{i=1}^{N+1} s_i \log{z_i} - z_i
\ee
We note that the profit function is the same as the log-likelihood function with the constant term omitted. Now, if the ${\pmb x_i} \in  \mR^N$ do not lie on an $(N-1)$-dimensional hyperplane, then clearly the global maxima of (\ref{cost}) is a finite set including ${\pmb y}={\pmb y^*}$ with the consequence that gradient ascent maximization of (\ref{cost}) is a candidate localization algorithm. 
Suppose ${\pmb y}[k]$ is the current estimate of ${\pmb y^*}$, and sufficiently small $\mu>0$, such an algorithm will proceed as:
\be \la{grad}
{\pmb y}[k+1]={\pmb y}[k]+\mu \left .  \frac{\partial J({\pmb y})}{\partial {\pmb y}}  \right |_{{\pmb y}={\pmb y}[k]}   \forall k\geq k_0.
\ee
Whether we can estimate ${\pmb y}= {\pmb y^*}$ uniquely is the subject of discussion in section \ref{smain}.

\section{The main result}  \la{smain}
This section presents the main results in this paper starting with the following lemma.
\begin{lemma} \la{lconv}
Suppose $\{{\pmb y},{\pmb y^*}\}\in \mR^N$ obey: ${\pmb y}\neq {\pmb y^*}$\\
 then  $S=\left \{{\pmb \eta} \in \mR^N \left | \|{\pmb \eta}-{\pmb y}\| \leq {\pmb \eta}-{\pmb y^*} \|\right  .   \right \}$  defines an open half plane in $\mR^N$ with a separating hyperplane  $$\mathcal{H}=\left \{{\pmb \eta} \in \mR^N \left | \|{\pmb \eta}-{\pmb y}\| =  \|{\pmb \eta}-{\pmb y^*} \|\right  .   \right \}.$$ 
\end{lemma}
\begin{proof}
Without loss of generality, suppose ${\pmb y^*} =0$. This can be attained through translation and rotation because distance measurements are invariant under translation and rotation. We abuse notation and maintain the variables ${\pmb y}$ and ${\pmb y^*}$ to preserve clarity. $\Leftrightarrow$ 
\bea \la{hpsg}
\begin{split}
({\pmb \eta} -{\pmb y})^T ({\pmb \eta} -{\pmb y}) \leq \eta^T \eta \\
2{\pmb \eta}^T {\pmb y} - {\pmb y}^T {\pmb y} \geq 0
\end{split}
\eea 
(\ref{hpsg}) is an open half plane in $\mR^N$ with a separating hyperplane of $({\pmb \eta}-\frac{1}{2}{\pmb y})^T {\pmb y} = 0$ in $\mR^{N-1}$. 
\end{proof}

\begin{theorem} \la{tmain1}
Suppose $\{{\pmb y},{\pmb y^*}\} \in \mR^N$ obeys: 

(a) ${\pmb y}\neq {\pmb y^*}$ 

(b) $ S=\left \{{\pmb \eta} \in \mR^N \left | \|{\pmb \eta}-{\pmb y}\| \leq \|{\pmb \eta}-{\pmb y^*} \|\right  .   \right \}$ and 

(c) ${\pmb y} \in S$. 

Then  ${\pmb y^*} \notin S$.
\end{theorem}
\begin{proof}
See \cite{BDBM} for proof.
\end{proof}

\begin{theorem} \la{tmain2}
Under Assumptions \ref{an} and \ref{aconv}, suppose $\{{\pmb y},{\pmb y^*}\} \in \mR^N, i \in \{1,\cdots,N+1\}$ obeys: 

(a) ${\pmb y^*} \in co\{{\pmb x_1},\cdots,{\pmb x_{N+1}}\}$ 

(b) $\{\pmb y^*\} \neq \{\pmb y\}$ 

Then there exists an $\{i,j\}$ pair such that  $\|{\pmb x_i} - {\pmb y}\| < \|{\pmb x_i} - {\pmb y^*}\| $  and  $\|{\pmb x_j} - {\pmb y}\| < \|{\pmb x_j} - {\pmb y^*}\| $.
\end{theorem}

\begin{proof} 
See \cite{BDBM} for proof.
\end{proof}

\begin{theorem} \la{tmain}
Consider (\ref{cost}) under assumptions \ref{amon}, \ref{an} and \ref{aconv}. Consider ${\pmb y} \in   co\{{\pmb x_1},\cdots,{\pmb x_{N+1}}\}$. Then there holds: 
\[
\frac{\partial J({\pmb y})}{\partial {\pmb y}} =0 \Leftrightarrow 
J({\pmb y})= \underset{{\pmb y}}{\operatorname{argmax}} \left(J({\pmb y})\right)\] 

Further, ${\pmb y}^*$ is unique.
\end{theorem}

\begin{proof}
Since ${\pmb y} \in   co\{{\pmb x_1},\cdots,{\pmb x_{N+1}}\}$, there exist $\beta_{i}>0$, $\forall i \in \{1,\cdots,N+1\}$, such that 
\be \la{bsum}
\sum_{i=1}^{N+1}\beta_{i}=1, ~ \sum_{i=1}^{N+1}\beta_{i} {\pmb x_i}={\pmb y}.
\ee 
From (\ref{bsum}) we obtain: 
\[
\sum_{i=1}^{N+1}\beta_{i} {\pmb x_i}=\left (   \sum_{i=1}^{N+1}\beta_{i}  \right ){\pmb y}
\Leftrightarrow \sum_{i=1}^{N+1}\beta_{i} ({\pmb x_i}-{\pmb y})=0.
\] 
In other words ${\pmb \beta}=[\beta_{1},\cdots, \beta_{N+1}]^\top$ is in the right nullspace of the matrix:
$
{\pmb {\cal X}}({\pmb y})=\left [\begin{matrix}    {\pmb x_1}-{\pmb y} & \cdots &{\pmb x_{N+1}}-{\pmb y}   \end{matrix}\right ].
$
As the ${\pmb x_i}$'s do not lie on an $(N-1)$-dimensional hyperplane ${\pmb {\cal X}}({\pmb y})$ has rank $N$  for all ${\pmb y}\in \mR^N$. Thus its nullspace has dimension 1, and as $\beta_i>0$, all its non-zero null vectors have elements that are either all positive, or are all negative. Now suppose 
\be \la{part0}
\frac{\partial J({\pmb y})}{\partial {\pmb y}} =0.
\ee 
Notice that (\ref{yml}) can be re-written so that the gradient of the profit function becomes: 
\begin{equation} \la{expgrad}
\frac{\partial J({\pmb y})}{\partial {\pmb y}} =\left. \sum\limits_{i=1}^{N+1} \left( \frac{s_i}{\frac{e^{-\aa d_i}}{d_i^2}}-A\right) \left(\aa + \frac{2}{d_{i}} \right)\frac{e^{-\aa d_{i}}}{d_{i}^3} \left({\pmb x}_i-{\pmb y} \right)\right .
\end{equation}
Define: 
\be \la{eta}
\xi_{i}= \left( \frac{s_i}{\frac{e^{-\aa d_i}}{d_i^2}}-A\right) \left(\aa + \frac{2}{d_{i}} \right)\frac{e^{-\aa d_{i}}}{d_{i}^3} 
\ee 
and notice that we can rewrite (\ref{expgrad}) to obtain
\begin{equation} \la{gradient1}
\begin{split}
\frac{\partial J({\pmb y})}{\partial {\pmb y}} = \sum_{i=1}^{N+1} \xi_{i} ({\pmb x_i}-{\pmb y}) 
\end{split}
\end{equation}
From  (\ref{gradient1}) ,  ${\pmb \xi}=[\xi_{1},\cdots, \xi_{N+1}]^\top$ is in the null space of ${\pmb {\cal X}}({\pmb y})$ if $\frac{\partial J({\pmb y})}{\partial {\pmb y}} =0$. 
Now suppose ${\pmb \xi}\neq {\pmb 0}$. Then every $\xi_{i}$ is either positive or every one of them is negative. Suppose $\xi_{i} > 0, i \in \{1, \cdots, N+1\}$.
This implies that 
\begin{align} \label{uniq}
\begin{split}
\frac{s_i}{\frac{e^{-\aa d_i}}{d_i^2}}-A &>0 ~\forall i \in \{1,\cdots, N+1\}
\end{split}
\end{align} 
Now suppose there are two solutions, ${\pmb y^*}$ and ${\pmb y} $, then (\ref{uniq}) can be written as:
\begin{align} \label{uniq1}
\begin{split}
A^* & \left( \frac{e^{-\aa d_i^*}}{d_i^{*2}}    \frac{d_i^2} {e^{-\aa d_i}} - 1 \right) > 0  \\
\frac{e^{-\aa d_i^*}}{d_i^{*2}} & > \frac{e^{-\aa d_i}}{d_i^2} ~\forall i \in \{1,\cdots, N+1\}
\end{split}
\end{align} 
(\ref{uniq1}) therefore reduces to:
\begin{equation} \label{uniq2}
\frac{e^{-\aa \|{\pmb x_i} -{\pmb y} \|}}{\|{\pmb x_i} -{\pmb y} \|} < \frac{e^{-\aa \|{\pmb x_i} -{\pmb y^*} \|}}{\|{\pmb x_i} -{\pmb y^*} \|} ~\forall i~\in \{1,\cdots, N+1\}
\end{equation} 
The strictly decreasing nature of $\frac{e^{-\aa \|{\pmb x_i} -{\pmb y} \|}}{\|{\pmb x_i} -{\pmb y} \|} \forall i$ implies:
\begin{equation} \label{uniq3}
\|{\pmb x_i} -{\pmb y} \| > \|{\pmb x_i} -{\pmb y^*} \| ~ \forall i ~\in \{1,\cdots, N+1\}
\end{equation} 
 Now since both $\{{\pmb y},{\pmb y^*}\} \in co\{{\pmb x_1}, \cdots, {\pmb x_{N+1}}\}$, the saparating hyperplane theorem precludes (\ref{uniq3}). 
Hence $\eta_{i} \not >0 \forall i ~\in \{1,\cdots, N+1\}$. A similar argument can be made to show that $\eta_{i} \not <0  \forall i~\in \{1,\cdots, N+1\}$. This means $\eta_{i} =0 \forall i~\in \{1,\cdots, N+1\}$.  
Further, suppose that $\eta_{i} =0 ~\forall i~\in \{1,\cdots, N+1\}$, a similar analysis will lead to  \[ \|{\pmb x_i} -{\pmb y} \| = \|{\pmb x_i} -{\pmb y^*} \| ~\forall i ~\in \{1,\cdots, N+1\}\]  leading to the uniqueness of the solution. 
Hence  \[
\frac{\partial J({\pmb y})}{\partial {\pmb y}} =0  \Leftrightarrow  
J({\pmb y})= \underset{{\pmb y}}{\operatorname{argmax}} \left(J({\pmb y})\right)\]  
This concludes the proof.
\end{proof}
\begin{remark}
The result is counter intuitive; that we can still localize uniquely with exactly  $N+1$ having an additional unknown variable $A$. We can illustrate this result to be true in the 1-D case for the conventional RSS model where path loss coefficient $\alpha=0$ in (\ref{rss}). In this case, two sensor measurements is sufficient to localize the source location even though there is an additional unknown parameter $A$. Consider Figure \ref{fig:1d}.
\begin{figure}[!htb]
\begin{center}
\includegraphics[scale=.4]{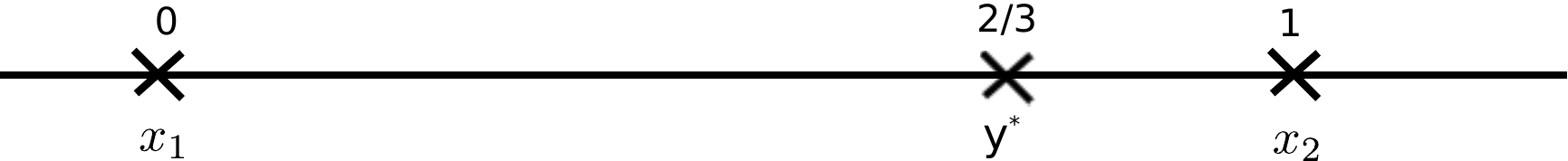}
\caption{Illustration of a source in $co\{x_1,x_2\}$}
\label{fig:1d}
\end{center}
\end{figure}
With $\alpha =0$ and arbitrary $A$, the measurements at the two sensors are $s_1 = 2.25A$ and $s_2= 9A$ which  will result in the  maximum likelihood equations: 
\bea \la{1deq1}
(1-y)^2= 0.25 y^2, ~
0 < y<1
\eea 
From (\ref{1deq1}), $y \in \{\frac{2}{3},2\}$ and the restriction on  $y$ ensures only $\frac{2}{3}$ is the unique admissible maximum likelihood solution.  
\end{remark}
Consequently, our results show that if the source is  $co\{{\pmb x_1}, \cdots, {\pmb x_{N+1}}\}$, there is only 1 optimal solution and $co\{{\pmb x_1}, \cdots, {\pmb x_{N+1}}\}$ is without a false maxima. Localization is therefore guaranteed provided the location estimate never leaves $co\{{\pmb x_1}, \cdots, {\pmb x_{N+1}}\}$. 

It is however conceivable that the solution at an epoch of iteration of the algorithm may leave $co\{{\pmb x_1}, \cdots, {\pmb x_{N+1}}\}$. We follow the approach in \cite{BDBM} to tackle the instances when the estimate leaves $co\{{\pmb x_1}, \cdots, {\pmb x_{N+1}}\}$. The argument in \cite{BDBM} that such a projection based algorithm will converge in probability as long as the source is inside  $co\{{\pmb x_1}, \cdots, {\pmb x_{N+1}}\}$ follows. In practice simulations presented in \ref{ssim}, the estimates rarely leave $co\{{\pmb x_1}, \cdots, {\pmb x_{N+1}}\}$.

\section{Simulations}  \la{ssim}
Two simulation scenarios for $\mathbb{R}^N,~N \in \{2,3\}$ are considered. The received signal at sensor $i$ is:  $s_{i}\sim \mbox{Poisson}\left(\frac{A}{d_i^2} e^{-\alpha d_i} +w_i \right)$ and
$w_i$ is the  background noise at sensor $i$. The $SNR$ is computed as: 
%$$SNR = 10\log_{10}\left (\frac{\sum_{i=1}^{N+1} \frac{A e^{-\alpha d_i}}{d_i^2}}{\sum_{i=1}^{N+1}w}\right)=10\log_{10}\left (\frac{1}{w(N+1)} {\sum_{i=1}^{N+1} \frac{A e^{-\alpha d_i}}{d_i^2}}\right).$$ 
$$SNR =10\log_{10}\left (\frac{A}{\sum_{i=1}^{N+1} w_i} {\sum_{i=1}^{N+1} \frac{e^{-\alpha d_i}}{d_i^2}}\right).$$ 
In all cases $\aa=0.0068$, $A=5\times 10^7$, $\mu=10^{18}$. The root mean squared error (RMSE) is averaged over 10000 random initial start points all within $co\{{\pmb x_1}, \cdots, {\pmb x_{N+1}}\}$. The algorithm runs for no more than 500 iterations. Also, for each iteration of each run, the $s_i$'s are generated independently. A projection augmented gradient ascent maximization of (\ref{cost}) under (\ref{orig}) is performed. The fact that the actual $s_i$'s differ from the value used in generating the gradient, confirms the robustness of the algorithm to uncertainties in the $s_i$ with unknown $A$.

Fig. \ref{fig:2Dsnrall} shows the  performance  when $N=2$. The sensors are located at (0,0), (200,0) and (50,200), the source is at (120,40), $co\{{\pmb x_1}, \cdots, {\pmb x_{N+1}}\}$. Fig. \ref{fmap} presents the map of the average location estimate provided by our algorithm for various SNR values, as well as the actual source location. 

\color{black}
Fig. \ref{ftime} is for gauging the convergence speed. For an SNR of 16.5dB,  Fig. \ref{ftime} plots the RMSE as a function of the iteration index $k$. The RMSE at each value of $k$ is obtained by averaging over the 10000 random runs described above. The fast rate of convergence is self-evident.
%\vspace{-0.1 in}
\begin{figure}[!htb]
\centering
\begin{subfigure}{.25\textwidth}
  \centering
\includegraphics[width=1\linewidth]{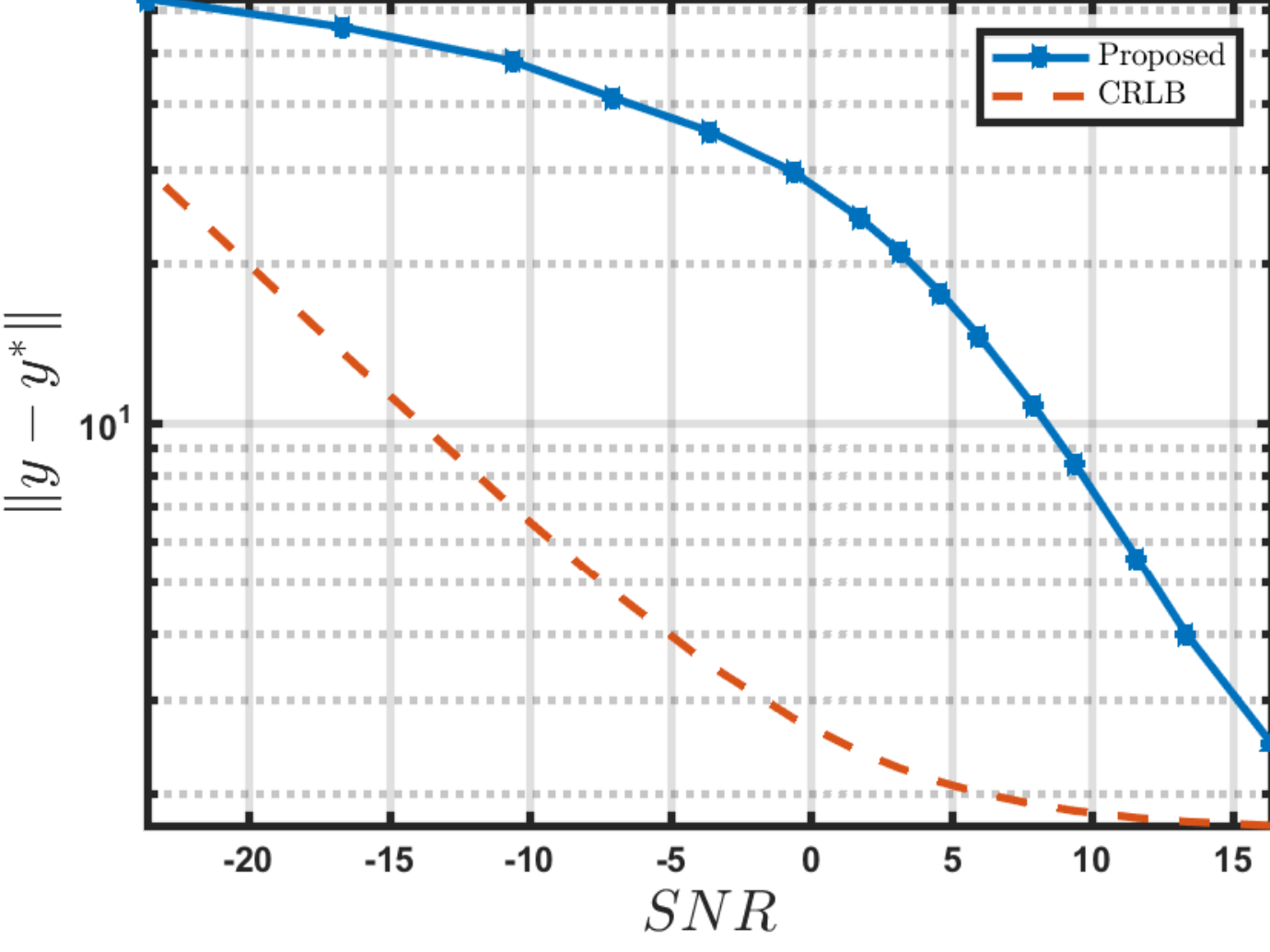}
\caption{ RMSE against SNR.}
  \label{fig:2Dsnrall}
\end{subfigure}%
\begin{subfigure}{.25\textwidth}
  \centering
  \includegraphics[width=1\linewidth]{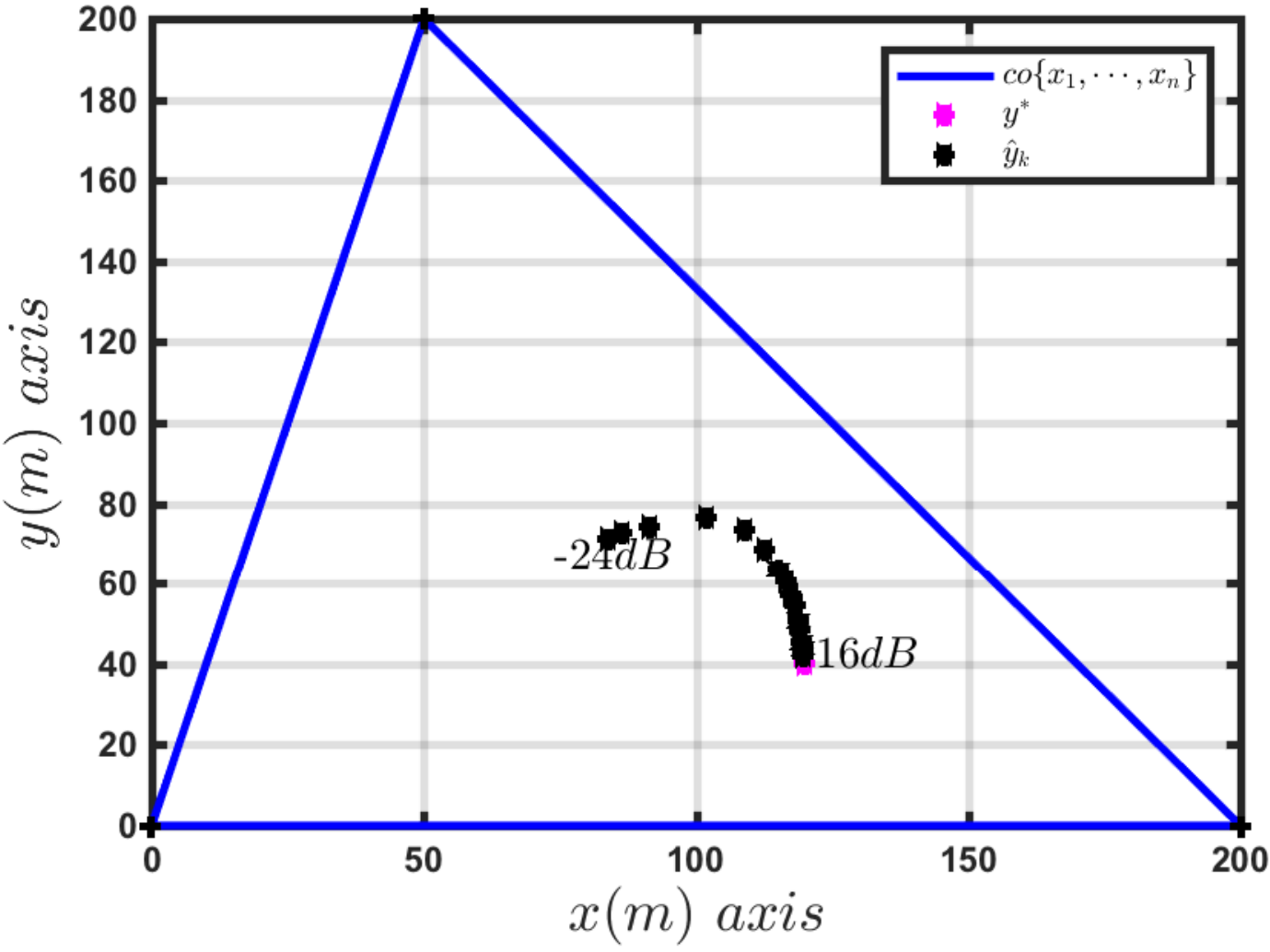}
  \caption{Average source location}
  \label{fmap}
\end{subfigure}
\caption{Simulations for $N=2$}

\label{fig:test}
\end{figure}

\begin{figure}[!htb]
\centering
\begin{subfigure}{.25\textwidth}
  \centering
\includegraphics[width=1\linewidth]{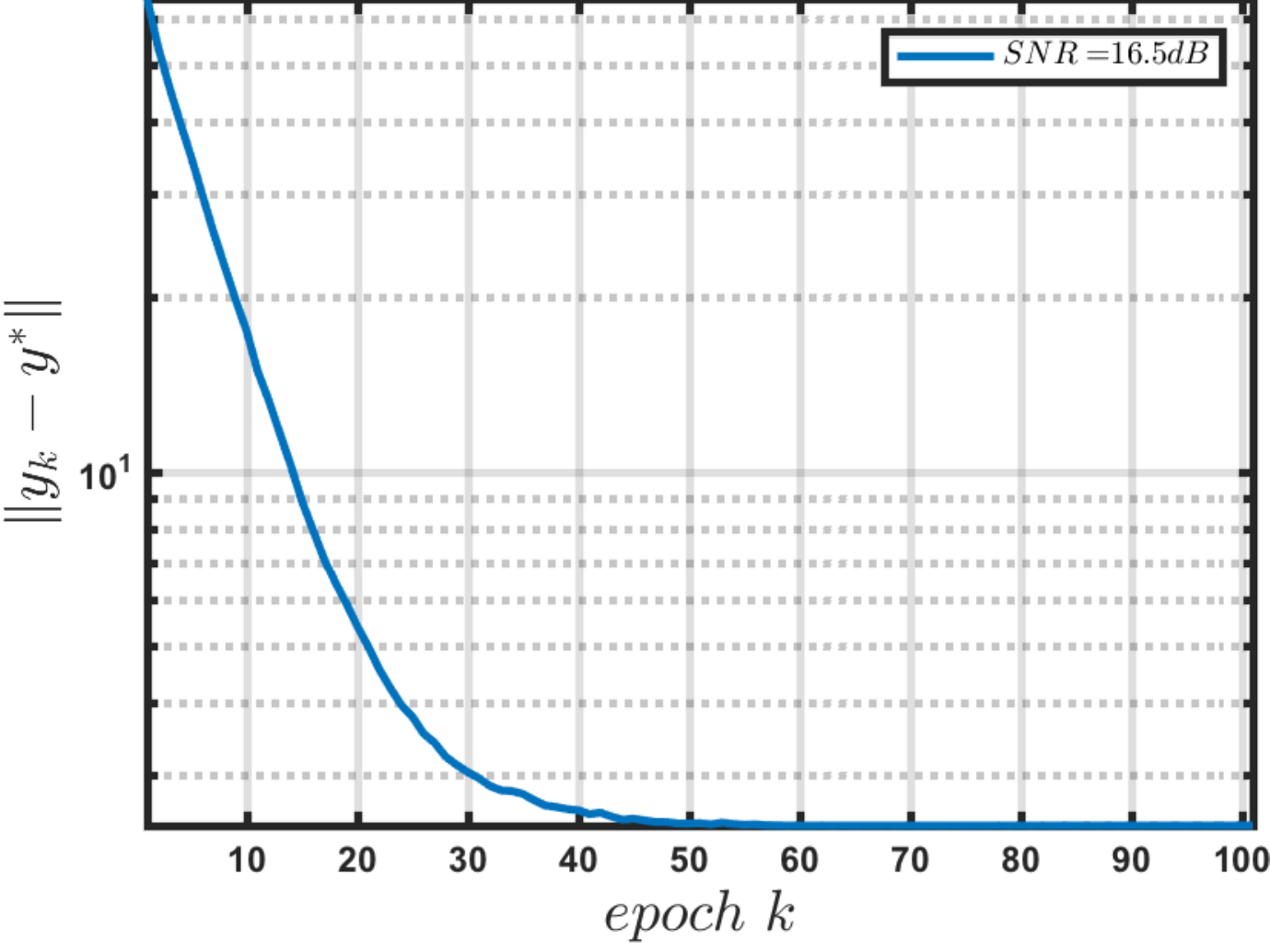}
\caption{ RMSE vs $epoch~k$}
  \label{ftime}
\end{subfigure}%
\begin{subfigure}{.25\textwidth}
  \centering
  \includegraphics[width=\linewidth]{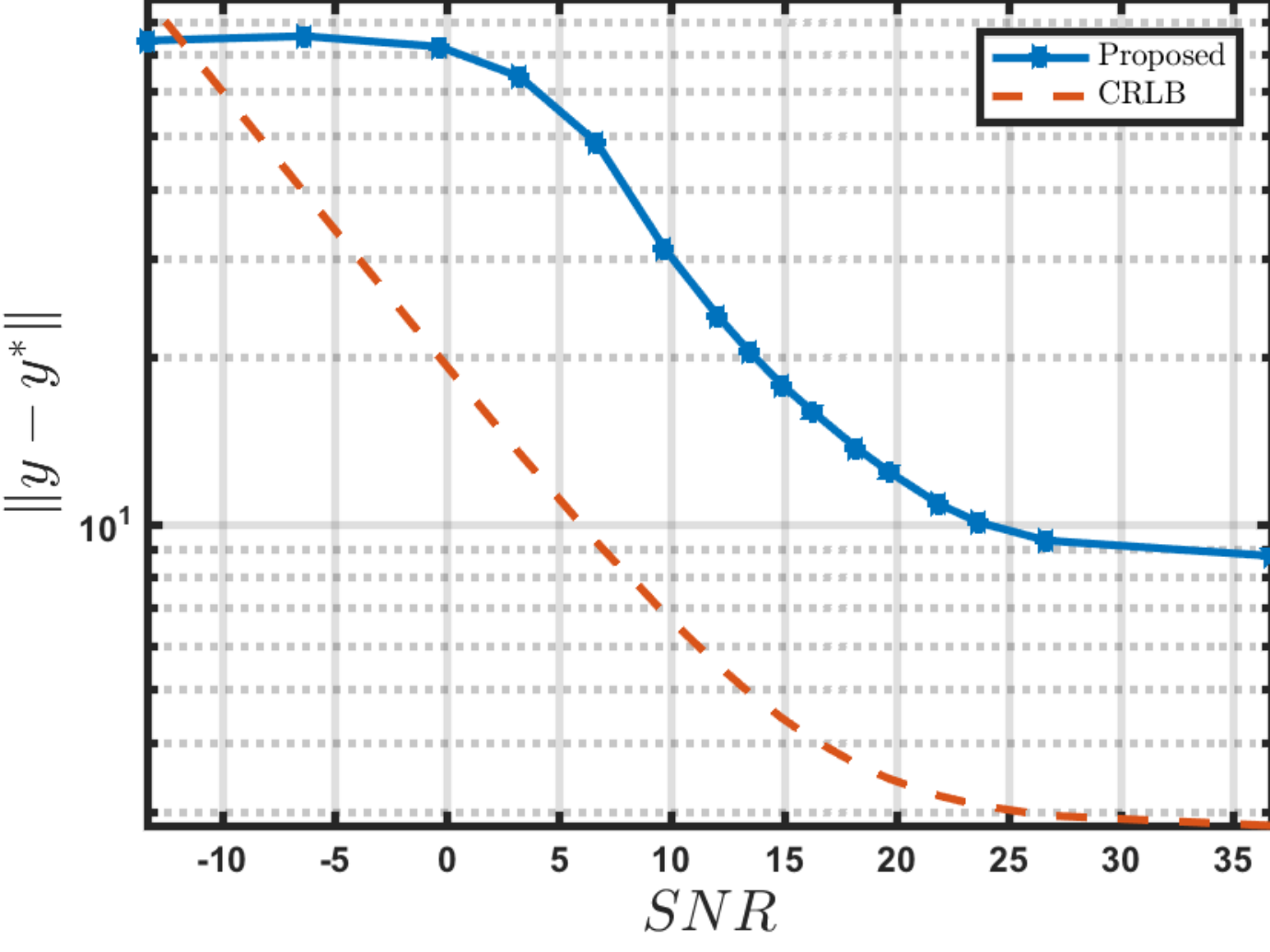}
  \caption{RMSE vrs SNR, $N=3$}
  \label{fig:3Dsnrall}
\end{subfigure}
\caption{Average converging rate and $\mathbb{R}^3$ simulations.}

\label{fig:test}
\end{figure}

%\vspace{+.10 in}
The 3-dimensional counterpart of  Fig. \ref{fig:2Dsnrall} is depicted in Fig. \ref{fig:3Dsnrall}. The  sensors are at (0,0,0), (200,0,0), (0,200,0) and (0,0,100) and the source is at (10,20,10). The performance is depicted in figure \ref{fig:3Dsnrall}.

It should be noted that in these simulations, except in low SNR regimes, the estimates do not leave $co\{{\pmb x_1}, \cdots, {\pmb x_{N+1}}\}$. Even with low SNRs they leave the $co\{{\pmb x_1}, \cdots, {\pmb x_{N+1}}\}$ only  about $10^{-3}$\% of times. Further, the expectation of the signal at the sensors in the absence of noise are very small. For $N=2$ the received signals in the absence of noise are $[1276, 3192, 490]^T$ whiles those for $N=3$ are $[7054, 372 ,449,3095]^T$. The simulations show the robustness of the algorithm.

\section{Conclusion}\la{scon}

A projection based  gradient ascent localization  of radioactive sources has been presented. It has been shown that if the source lies in the $co\{{\pmb x_1}, \cdots, {\pmb x_{N+1}}\}$, then the maximum likelihood estimates has no false stationary points. The algorithm is proved to achieve global uniform asymptotic convergence in probability with simulations demonstrating robustness of algorithm.


\nocite{FDA} 
\nocite{hero} 
\nocite{8auth} 
\nocite{yee2011comparison}
\nocite{6485009}
\nocite{rao2008localization}
\nocite{howse2001least}
\nocite{gunatilaka2007localisation}
\nocite{morelande2009radiological}
\nocite{5779060}
%\begin{thebibliography}{1}
\bibliographystyle{IEEEtran}
\bibliography{refs.bib} 

\end{document}